\documentclass[10pt, a4paper]{amsart}
\usepackage{amssymb, amsthm}  
\usepackage{graphicx, psfrag}

\theoremstyle{plain}
    \newtheorem{thm}{Theorem}

    \newtheorem*{scho}{Scholium}
    \newtheorem{prop}{Proposition}

    \newtheorem{lemma}[prop]{Lemma}

\theoremstyle{definition}
    \newtheorem*{defn}{Definition}
\theoremstyle{remark}
    \newtheorem*{rem}{Remark}
    \newtheorem*{ack}{Acknowledgements}
     \newtheorem*{claim}{Claim}

\newcommand{\R}{\mathbb{R}}

\newcommand{\Z}{\mathbb{Z}}
\newcommand{\N}{\mathbb{N}}

\renewcommand{\P}{\mathbb{P}}
\newcommand{\cA}{\mathcal{A}}\newcommand{\cC}{\mathcal{C}}\newcommand{\cD}{\mathcal{D}}

\newcommand{\cP}{\mathcal{P}}

\renewcommand{\setminus}{\smallsetminus}
\renewcommand{\emptyset}{\varnothing}
\renewcommand{\Im}{\mathrm{Im}}
\newcommand{\SL}{\mathrm{SL}}
\newcommand{\GL}{\mathrm{GL}}

\newcommand{\Id}{\mathrm{Id}}

\DeclareMathOperator{\interior}{int}

\newcommand{\wed}{{\mathord{\wedge}}}

\newcommand{\m}{\mathfrak{m}}

\title{Some Characterizations of Domination}

\author[Bochi]{Jairo Bochi}
\address{Catholic University of Rio de Janeiro, Brazil}
\urladdr{www.mat.puc-rio.br/$\sim$jairo}
\email{jairo@mat.puc-rio.br}

\author[Gourmelon]{Nicolas Gourmelon}
\address{IMPA, Rio de Janeiro, Brazil}
\email{nicolas@impa.br}

\date{August 31, 2008.}
\thanks{Both authors were partially supported by CNPq (Brazil).}

\begin{document}

\begin{abstract}
We show that a cocycle has a dominated splitting if and only if there is a uniform exponential gap between
singular values of its iterates.
Then we consider sets $\Sigma$ in $\GL(d,\R)$ with the property that
any cocycle with values in $\Sigma$ has a dominated splitting.
We characterize these sets in terms of existence of invariant multicones,
thus extending a $2$-dimensional result by Avila, Bochi, and Yoccoz.
We give an example showing how these multicones can fail to have convexity properties.
\end{abstract}

\maketitle

\section{Introduction}

Let $X$ be a compact invariant set for a diffeomorphism $T$,
and $E\oplus F$ is a splitting  of the tangent bundle over $X$ that is invariant by the tangent map $dT$.
Following Smale, the splitting is called \emph{hyperbolic} if 
vectors in $E$ are uniformly expanded by $dT$, while vectors in $F$ are uniformly contracted.
This notion is very stringent, and many weaker forms of it are studied in the literature.
One of these, that is the concern of this paper, is \emph{domination}.
A splitting $E\oplus F$ as above is called \emph{dominated} if
above each point all vectors in $E$ are more expanded than all vectors in $F$.
This is equivalent to $E$ being a hyperbolic attractor and $F$ being a hyperbolic repeller in the projective bundle.
Hence domination could also be called uniform projective hyperbolicity.

This notion was important in the works of
Ma\~n\'e~\cite{DD78_Mane} and Liao~\cite{DD80_Liao} on Smale's Stability Conjecture.
The term ``dominated splitting'' apparently was introduced by Ma\~n\'e in \cite{DD84_Mane}.
The concept of domination also appeared in Ordinary Differential Equations
under several different forms and names, see \cite{DD60, DD66, DD69, DD78bis, DD82}.
Among recent work in  Differentiable Dynamical Systems
involving domination, one can mention
\cite{BDP, DD02_Pujals, BV Annals, DD07_Pujals, DD07_Diaz}.

The concepts of hyperbolicity and domination naturally make sense
in the general setting of automorphisms of vector bundles, or {\em linear cocycles}.
In the case of bundles of the form $X \times \R^d$,
a linear cocycle is then canonically identified to a pair $(T,A)$, where $T$ is a dynamics on $X$ and $A$ is a family of matrices indexed by $X$. 
Such systems are widely studied in several contexts, as for example
products of random matrices and Schr\"odinger operators~\cite{BL}.

\medskip

Our aim is to give some characterizations of domination. Let us briefly summarize the results and the organization of this paper.

In~\cite{Yoccoz}, Yoccoz shows that a cocycle on a $2$-dimensional vector bundle 
admits a hyperbolic splitting if and only if
the non-conformality of the matrices of the iterates grows uniformly exponentially.
Our first result extends this to any dimension, giving 
a necessary and sufficient criterium for domination that does not refer to any splitting.
More precisely, we prove in Section~\ref{s.cocycles} that
a cocycle admits a dominated splitting $E \oplus F$  with $\dim E =i$ if and only if the $i$-eccentricity (the ratio between the $i$-th and $(i+1)$-th singular values) of the matrices of the $n$-th iterate increase uniformly exponentially as $n$ goes to $\infty$.

In Section~\ref{s.sets}, we deal with bundles of the form $X \times \R^d$.
We consider families of matrices $A=\{A_x\}_{x\in X}$ that are dominated in the sense that
the cocycle $(T,A)$ admits a dominated splitting for any choice of the dynamics $T$.
A study of dominated finite families of matrices in $\SL(2,\R)$ was initiated by~\cite{Yoccoz, ABY}.
An important tool is the characterization of these families by the existence of what they call \emph{invariant multicones}
in projective space.
Here we find a generalization of such a description to arbitrary dimension.

Much of the results of~\cite{ABY} rely on the simple structure of the multicones in dimension $2$.
This simplicity cannot be retrieved in greater dimension. 
Indeed, we show in Section~\ref{s.example} that the connected components of the multicones
may fail to have any convexity property.


\begin{ack}
We would like to thank Christian Bonatti for some important ideas and discussions.
Also thanks to Sylvain Crovisier and Lorenzo D\'iaz.
\end{ack}

\section{Dominated Splittings for Cocycles}\label{s.cocycles}

Let $V$ be a continuous vector bundle
over a compact Hausdorff space $X$,
and of finite dimension $d$.
Denote by $V_x$ the fibre over $x \in X$.

Fix on $V$ a (continuous) inner product $\langle \mathord{\cdot}, \mathord{\cdot} \rangle$,
and let $\| \mathord{\cdot} \|$ be the induced norm.
If $B: V_x \to V_y$ is a linear map, its norm $\|B\|$ and co-norm $\m(B)$ are defined respectively
as the supremum and the infimum of $\|Bv\|$ over the unit vectors $v \in V_x$.
We have $\m(B)= \|B^{-1}\|^{-1}$ when $B$ is invertible.
Let $B^*: V_y \to V_x$ be the adjoint of $B$.
We denote by $\sigma_1(B) \ge \cdots \ge \sigma_d(B)$
the singular values of $B$, that is,
the eigenvalues of $(B^* B)^{1/2}$.
We have $\sigma_1(B) =\|B\|$ and $\sigma_d(B) = \m(B)$.

Let $\cA: V \to V$ be a vector bundle automorphism,
fibering over a homeomorphism $T : X \to X$.
We also call $\cA$ a \emph{cocycle}.
For each $x\in X$ and $n \in \Z$, we let $A^n(x)$ be the
restriction of $\cA^n$ to the fibre $V_x$;
it is a linear map from $V_x$ to $V_{T^n(x)}$.
We write $A(x) = A^1(x)$.
In the case where the vector bundle is trivial, ie, $V = X \times \R^d$,
we can regard $A$ as a map $A \colon X \to \GL(d,\R)$;
by abuse of notation we write $\cA = (T,A)$.

A \emph{splitting} $E\oplus F$ for the bundle $V$
is a continuous family of splittings $E_x \oplus F_x = V_x$,
where we require the dimensions of $E_x$ and $F_x$ to be constant.
The dimension of $E$ is called the \emph{index} of the splitting.
The splitting is invariant (with respect to $\cA$) if
$A(x) \cdot E_x = E_{T(x)}$, $A(x) \cdot F_x = F_{T(x)}$,

A splitting $V = E \oplus F$ is called
\emph{dominated} for $\cA$ if it is invariant and
there are constants $C>0$ and $0<\tau<1$ such that
$$
\frac{\| A^n(x)|F_x \|}{\m (A^n(x)|E_x)} < C \tau^n \quad
\text{for every $x\in X$ and every $n \ge 0$.}
$$
It is also said that \emph{$E$ dominates $F$}.
It is always possible to find an {\em adapted metric}, that is, an inner product such that $C=1$, see~\cite{G}.\footnote{The existence of adapted metrics justifies the informal definition of the introduction.}
Some properties about dominated splittings are collected in \cite{BDV_livro}.

\medskip

Our first result is:

\begin{thm}\label{t.yoccoz}
The following assertions about a vector bundle automorphism $\cA$ are equivalent:
\begin{enumerate}
\item There is a dominated splitting of index $i$.
\item There exist $C>0$ and $\tau<1$ such that
${\displaystyle \frac{\sigma_{i+1}(A^n(x))}{\sigma_i(A^n(x))} < C \tau^n}$
for all $x\in X$ and $n \ge 0$.
\end{enumerate}
\end{thm}

Theorem~\ref{t.yoccoz} was already proved in dimension $2$
by Yoccoz in \cite{Yoccoz}.

\begin{rem}
One easily extends the theorem to continuous-time cocycles $t\in \R\to\cA^t$, replacing the $n$-th iteration $A^n(x)$ by the time-$n$ cocycle.
\end{rem}

\begin{rem}
One can also consider the case where $T$ is non-invertible and the $A(x)$ are all invertible, considering the natural extension $\tilde{\cA}$ of $\cA$ that sends any orbit $(v_i)_{i\in \Z}$ of $\cA$ on its left shift $(v_{i+1})_{i\in \Z}$, and saying that $\cA$ admits a domination of index $i$ if and only if the invertible cocycle $\tilde{\cA}$ admits a domination of index $i$. From Theorem~\ref{t.yoccoz} and the fact that the products of matrices that appear in the iterates $\cA^n$ are the same as those in $\tilde{\cA}^n$ one easily deduces that the theorem also holds $\cA$.
\end{rem}

\subsection*{Proof}
First of all, let us see that the
validity of each condition in Theorem~\ref{t.yoccoz} does not depend on the choice of the inner product.
This independence is obvious for condition~(a), and for condition~(b) it is a consequence of the following lemma:

\begin{lemma}\label{l.exterior}
For every $d\in \N$ and $C_1>1$ there is $C_2>1$
such that
for any linear maps $A$, $N: \R^d \to \R^d$,
with $\|N^{\pm 1}\| \le C_1$, we have
$$
C_2^{-1} \le \frac{\sigma_k(NA)}{\sigma_k(A)} \le C_2 \quad \text{for each $k=1$, \ldots, $d$.}
$$
The same is true if we replace $NA$ by $AN$.
\end{lemma}

\begin{proof}
We will use a few facts about exterior powers; see e.g.~\cite{Arnold_RDS}.
Let $\wed^k \R^d$ indicate the $k$-th exterior power of $\R^d$.
If $A: \R^d \to \R^d$ is a linear map, then it induces a linear map
$\wed^k A: \wed^k \R^d \to \wed^k \R^d$.
The Euclidian inner product in $\R^d$ induces an inner product in $\wed^k \R^d$,
and the corresponding operator norm satisfies
$$
\|\wed^k A\| = \sigma_1 (A) \cdots \sigma_k(A) \, .
$$
Given another linear map $N : \R^d \to \R^d$, we have
$$
\|N^{-1}\|^{-k} \, \|\wed^k A\| \le  \|\wed^k (NA)\| \le \|N\|^k \, \|\wed^k A\|
$$
The lemma follows.
\end{proof}

Now we prove separately the two implications of Theorem~\ref{t.yoccoz}:

\begin{proof}[Proof that (a) implies (b)]
First assume $E \oplus F$ is a dominated splitting of index~$i$.
As an immediate consequence of the definition, the angle between $E$ and $F$ is bounded from below.
By making a continuous change of the inner product
(which is allowed due to the remarks above),
we can assume that $E_x$ and $F_x$ are orthogonal at every point $x \in X$.
Then for every point $x$, the map $A(x)^*$ sends the subspaces $E_{Tx}$ and $F_{Tx}$
respectively to $E_x$ and $F_x$.
Hence for any $n$,
the eigenvectors of $[(A^n(x))^* A^n(x)]^{1/2}$
belong to $E_x \cup F_x$.
These are of course the singular directions of $A^n(x)$.
Once the ratio
$$
\frac{\|A^n(x)| F_x\|}{\m(A^n(x)|E_x)}
$$
becomes bigger than $1$, the singular directions of $A^n(x)$
corresponding to the $i$ biggest singular values (counted with multiplicity)
are contained in $E_x$, and the others are contained in $F_x$.
In particular, the ratio above equals $\sigma_i(A^n(x))/ \sigma_{i+1}(A^n(x))$.
Condition (b) follows.
\end{proof}

Despite the purely topological nature of Theorem~\ref{t.yoccoz},
our proof makes use of the measure-theoretic Theorem of Oseledets; see e.g.~\cite{Arnold_RDS}.
It asserts that there is a set $R$ of full probability
(i.e., $R$ is a Borel set and $\mu(R)=1$ for every $T$-invariant Borel probability measure $\mu$)
such that for all $x \in R$ the following holds:
\begin{itemize}
\item For every non-zero vector $v \in V_x$, the following limits exist:
\begin{align*}
\lambda^+(x,v) = \lim_{n \to + \infty} \frac 1n \log \|A^n(x) \cdot v\|, \\
\lambda^-(x,v) = \lim_{n \to - \infty} \frac 1n \log \|A^n(x) \cdot v\|.  
\end{align*}
Moreover the sets of values attained by both are the same and form a finite set
$\{\chi_1(x) > \cdots  > \chi_k(x) \}$, where $k=k(x)$.

\item There is a splitting
$V_x = E^1_x \oplus \cdots \oplus E_x^k$ such that for each non-zero $v\in V_x$,
\begin{align*}
\lambda^+(x,v) \le \chi_j(x) \quad &\text{iff} \quad v \in E^j_x \oplus \cdots \oplus E^k_x \, , \\
\lambda^-(x,v) \ge \chi_j(x) \quad &\text{iff} \quad v \in E^1_x \oplus \cdots \oplus E^j_x \, .
\end{align*}

\item The linear maps $[A^{n}(x)^* A^n(x)]^{1/2n}$ converge as $n \to + \infty$
to a linear map whose eigenvalues
are $\exp \chi_1(x) > \cdots > \exp \chi_k(x)$.
Letting $\hat{E}_x^1$, \ldots, $\hat{E}_x^k$ denote the respective eigenspaces,
we have, for each $j = 1,\ldots, k$,
$$
\hat{E}^j_x \oplus \cdots \oplus \hat{E}^k_x =  E^j_x \oplus \cdots \oplus E^k_x \, .
$$
The analogous property for negative time also holds.
\end{itemize}
The points $x$ in $R$ are called \emph{regular},
the numbers $\chi_j(x)$ are called the \emph{Lyapunov exponents};
and $E^j_x$ are called the \emph{Oseledets spaces}.
The \emph{multiplicity} of the exponent $\chi_j(x)$ is defined as the dimension of the corresponding space $E^j_x$.
We rewrite the Lyapunov exponents repeated according to multiplicity as
$\lambda_1(x) \ge \cdots \ge \lambda_d(x)$.
Notice that if $x$ is regular then
$$
\lim_{n \to +\infty} \frac1n \log \sigma_i(A^n(x)) = \lambda_i(x).
$$

We can now end the proof of Theorem~\ref{t.yoccoz}:

\begin{proof}[Proof that (b) implies (a):]
Assume condition~(b) holds for some fixed~$i$.
For each $x \in X$ and every sufficiently large $n$,
we have $\sigma_i(A^n(x))>\sigma_{i+1}(A^n(x))$.
Then let $U^{(n)}_x$ be the eigenspace of $[(A^n(x))^* A^n(x)]^{1/2}$ associated to the singular values
$\sigma_1$, \ldots, $\sigma_i$, and let $S^{(n)}_x$ be the eigenspace associated to the other singular values.
These two spaces are orthogonal.

\begin{claim}
The sequence of spaces $S^{(n)}_x$ converges to some $F_x$, uniformly with respect to $x \in X$.
\end{claim}

\begin{proof}
Let us estimate the angle $\alpha_n(x)$ between $S^{(n)}_x$ and $S^{(n+1)}_x$.
Let $v \in S^{(n)}_x$ be the farthest unit vector from $S^{(n+1)}_x$.
Let $w$ be the projection of $v$ on $U^{(n+1)}_x$ along $S^{(n+1)}_x$.
Since the spaces $S^{(n+1)}_x$ and $U^{(n+1)}_x$ are orthogonal and so are their images by $A^{n+1}(x)$, we have
$\|w\| = \sin \alpha_n(x)$ and
$$
\|A^{n+1}(x) \cdot v\| \ge \|A^{n+1}(x) \cdot w\| \ge \sigma_i(A^{n+1}(x)) \, \|w\|.
$$
On the other hand,
$$
\|A^{n+1}(x) \cdot v\| \le \|A(T^n x)\| \cdot \|A^n(x) \cdot v\| \le \|A\| \sigma_{i+1}(A^n(x)).
$$
Thus
$$
\sin \alpha_n(x) = \|w\| \le \frac{\|A\| \, \sigma_{i+1}(A^n(x))}{\sigma_i(A^{n+1}(x))}
$$
Now using the assumption~(b) and Lemma~\ref{l.exterior}, we see
that there exists $C_1>0$ such that $\alpha_n(x) < C_1 \tau^n$ for all $x \in X$ and $n \ge 0$.
The claim follows.
\end{proof}

Now assume $x$ is a regular point.
Then, by assumption~(b)
there is a gap between the Lyapunov exponents:
$\lambda_i(x) -\lambda_{i+1}(x) \ge \log \tau^{-1}$.
Take $j=j(x)$ with $1 \le j \le k=k(x)$ such that $\chi_j(x) = \lambda_i(x)$.
It follows from the the definition of the  spaces $S^{(n)}_x$ that
their limit $F_x$ equals
$E^{j+1}_x \oplus \cdots \oplus E^k_x$.
In particular $\lambda^+(x,f) \le \chi_{j+1}(x) = \lambda_{i+1}(x)$
for every non-zero vector $f \in F_x$.

Next, notice that assumption (b) implies
$$
\frac{\sigma_{d+1-i}(A^{-n}(x))}{\sigma_{d-i}(A^{-n}(x))} < C \tau^n, \quad
\text{for all $x\in X$ and $n \ge 0$.}
$$
Thus we can apply a symmetric argument iterating backwards
and conclude that for every $x \in X$ there is a space $E_x$ of dimension $i$,
depending continuously on $x$, that equals $E^1_x \oplus \cdots \oplus E^j_x$
whenever $x$ is a regular point.
If $e$ is a non-zero vector in $E_x$, then $\lambda^+(x,e)$ (as well as $\lambda^-(x,e)$)
is at least $\lambda_i(x)$.

Hence for every regular point $x$ and every non-zero vectors $e \in E_x$, $f\in F_x$,
$$
\lim_{n\to + \infty} \frac 1n \log \frac{\| A^n(x) \cdot f \|}{\| A^n(x) \cdot e \|}
\text{ exists and is at most $\log \tau$.}
$$
Now consider the continuous fibre bundle $\hat{X} \to X$ whose fibre over $x$
is $P(E_x) \times P(F_x)$, where the letter $P$ denotes projectivization.
Let $\hat{T}: \hat X \to \hat X$ be the obvious bundle map induced by $\cA$.
Let $\psi: \hat{X} \to \R$ be given by
$$
\psi(x, \bar e, \bar f) = \log \frac{\| A(x) \cdot f \|}{\| A(x) \cdot e \|}
$$
where $e$, $f$ are unit vectors in the directions of $\bar e$, $\bar f$, respectively.
The Birkhoff averages of $\psi$ under $\hat T$ are:
$$
\psi_n(x , \bar e, \bar f) = \frac 1n \sum_{m=0}^{n-1} \psi\circ\hat{T}^m (x, \bar e, \bar f)  =
\frac 1n \log \frac{\| A^n(x) \cdot f \|}{\| A^n(x) \cdot e \|} \, .
$$
Therefore for every $\hat T$-invariant probability measure $\nu$ in $\hat{X}$, the
sequence $\psi_n$ converges at $\nu$-almost every point,
and the limit is at most $\log \tau$.
Since $\hat{T}$ is a homeomorphism of the compact space $\hat X$,
and $\psi$ is continuous,
it follows from a standard Krylov-Bogoliubov argument
that given any $\tau<\kappa<1$, there exists $n_0$ such that
for every $n \ge n_0$ we have $\psi_n < \log \kappa$
\emph{uniformly} over $\hat{X}$.
For $n=n_0$, this means that
$$
\frac{\|A^{n_0}(x) | F_x\|}{\m(A^{n_0}(x) | E_x)} < \kappa \quad \text{for all $x \in X$}.
$$
This proves that $V = E \oplus F$ is
indeed a splitting, and is dominated for $\cA$.
\end{proof}

\section{Dominated Sets} \label{s.sets}

In this part, $\Sigma$ is a compact subset of $\GL(d,\R)$.

\begin{defn}
The set
$\Sigma$ is called \emph{dominated of index $i$} iff
there exist $C>0$ and $0<\tau<1$ such that for any finite sequence
$A_1,\ldots,A_N$ in $\Sigma$ we have $$\frac{\sigma_{i+1}(A_1 \cdots A_N)}{\sigma_i(A_1 \cdots A_N)} < C \tau^N \, .$$
where $\sigma_i(M)$ denotes, as before, the $i$-th singular value of the matrix $M$.
\end{defn}

An $i$-dimensional vector subspace of $\R^d$ is called an \emph{$i$-plane}.
The Grassmannian $G(i,d)$ is defined as the set of $i$-planes.
The linear group $GL(d,\R)$ naturally acts on $G(i,d)$. Given a subset $C \subset G(i,d)$, we denote by $\Sigma^*C$ the image of $C$ by the action of $\Sigma$, that is the subset $\{A \cdot E \; ; A\in \Sigma, E\in C\}$ of $G(i,d)$. A subset $C$ of $G(i,d)$ is said to be {\em (strictly) invariant} by $\Sigma$ if  (the closure of) $\Sigma^*C$ is included in (the interior of) $C$.

For any $C\subset G(i,d)$ we denote by $PC$ the subset of $P\R^d =  G(1,d)$ of directions that are included in some element of $C$.

\begin{thm} \label{t.char}
Let $\Sigma \subset \GL(d,\R)$ be a compact set. The following four statements are equivalent:
\begin{itemize}
\item[(a)] The set $\Sigma$ is dominated of index $i$.
\item[(b)] For any dynamics $T:X \to X$ and any map $A: X \to \GL(d,\R)$ whose image is contained in $\Sigma$,
the cocycle $(T,A)$ has a dominated splitting of index $i$.
\item[(c)] There exists a non-empty subset $C\subset G(i,d)$ that is strictly invariant by $\Sigma$, and there is a $(d-i)$-plane that is transverse to all elements of $C$.
\item[(d)] There exists a subset $C\subset P\R^d$ that is strictly invariant by $\Sigma$, that contains the projection $PE$ of some $i$-plane $E$ and that does not intersect the projection $PF$ of some $(d-i)$-plane $F$.
\end{itemize}
\end{thm}

\begin{scho}
The sets $C$ of conditions (c) and (d) can be chosen so that they have a finite number of connected components (respectively in $G(i,d)$ and $P\R^d$), and their closures are pairwise disjoint.
\end{scho}

Our results motivate the following definition,
corresponding to that given in~\cite{ABY} for $d=2$ and $i=1$:

\begin{defn}
A set $C$ that satisfies condition (c) and has finitely many connected components with pairwise disjoint closures is called an
{\em (unstable) multicone of index~$i$} for~$\Sigma$.
\end{defn}

\subsection*{Proof}
We will first prove  Theorem~\ref{t.char},
and then explain how the scholium follows.

The {\em universal $\Sigma$-valued cocycle}
is the linear cocycle $(T,A)$ on $X\times \R^d$ where $X$ is the set of $\Z$-indexed sequences in $\Sigma$, the dynamics $T\colon X \to X$ is the shift, and $A\colon X\to \GL(d,\R)$ maps $x=(x_k)_{k\in \Z}\in X$ to $x_0$.

It is easy to see that (b) is equivalent to:
\begin{itemize}
\item[(b')] The universal $\Sigma$-valued cocycle admits a dominated splitting of index~$i$.
\end{itemize}

By Theorem~\ref{t.yoccoz}, conditions (a) and (b') are equivalent.
We are going to prove that:
$$
\text{(c)} \ \Rightarrow \ \text{(d)} \ \Rightarrow \ \text{(b')} \ \Rightarrow \ \text{(c).}
$$

\begin{proof}[Proof that (c) implies (d)]
Let $C$ be as in (c). Then the projection $PC$ of $C$ on $P\R^d$ is the good candidate to satisfy (d).
Indeed, since  $C$ strictly invariant by $\Sigma$,
we have
$$
\overline{\Sigma^* PC} =
\overline{P (\Sigma^* C)} \subset P (\overline{\Sigma^* C}) \subset P (\interior C)
\subset \interior (PC).
$$
In other words, $PC$ as well is strictly invariant by $\Sigma$.
\end{proof}

If $a$, $b$, $c$, $d$ are distinct points in the extended real line $\R \cup \{\infty\}$,
then their cross-ratio is defined as the real number
$$
[a,b,c,d] = \frac{c-a}{b-a} \cdot \frac{d-b}{d-c} \, .
$$
By taking projective morphisms, the definition is extended to any projective line.

\begin{proof}[Proof that (d) implies (b')]
Let $C$ be a subset of $P\R^d$ that satisfies all the conditions of (d).
Let $(T,A)$ be the universal $\Sigma$-valued cocycle. Let $\cC$ be the cone field $X\times C$ and $\cC^n$ the $n$-th iterate of $\cC$ by $(T,A)$. Let $\cC^\infty$ be the intersection $\bigcap_{n\in \N}\cC^n$ and denote by $\cC^n_x$ the fibre of $\cC^n$ above $x$, for all $n\in \N\cup\{\infty\}$ .

\begin{claim}
For all $x\in X$ the fibre $\cC^\infty_x$
is a projective vector space.
\end{claim}

This claim and its proof are adapted from an idea of Christian Bonatti for diffeomorphisms.

\begin{proof}[Proof of the claim]
It suffices to prove that for any pair $u\neq v$ of points in $\cC^\infty_x$, the projective line $D$ spanned by them is contained in $\cC^\infty_x$.
We reason by contradiction. Assume that $D$ is not contained in $\cC^\infty_x$. Then by compactness we find two vectors $e\neq f$  in the boundary $\partial_D \cC^\infty_x$ of $\cC^\infty_x\cap D$ relative to $D$. By definition of $\cC^\infty$, we find sequences $e_n$, $f_n$ in $\partial_D \cC^n_x$ such that $e_n$ tends to $e$ and $f_n$ to $f$. Then the cross-ratio $[e_{n+1},e_n,f_n,f_{n+1}]$ tends to $\infty$ as $n$ tends to $\infty$. Taking the $n$-th preimage of these quadruplets by $(T,A)$, we find sequences $a_n$, $d_n$ in $\partial \cC^1$ and $b_n$, $c_n$ in $\partial \cC$  such that the cross-ratio $[a_n, b_n, c_n, d_n]$ is defined and tends to infinity.
This would imply in particular that the distance between $\partial \cC^1$ and $\partial \cC$ is zero, which is absurd, by compactness and strict invariance of the cone field.
The claim is proved.
\end{proof}

We write $\cC^\infty_x = PE_x$. Notice that the complement cone field $\cD = (X\times P\R^d) \setminus \cC$ is strictly invariant for the inverse cocycle $(T,A)^{-1}$. Therefore we build an invariant field of vector spaces $\cD^{\infty}_x =PF_x$.
The field $\cC$ contains an $i$-dimensional bundle and $\cD$ contains a $(d-i)$-dimensional bundle, therefore $E$ and $F$ are vector bundles of respective dimensions $i$ and $d-i$. Besides the angle between the two bundles is bounded from below by strict invariance of the cone fields.

Let $\Gamma^\epsilon$ be the cone around $E$ formed by the directions $v=v_E+v_F$ where $v_E\in E$, $v_F\in F$ and $\|v_F\|\leq \epsilon \|v_E\|$. For $\epsilon$ small enough, $\Gamma^\epsilon$ is in the interior of $\cC$. By definition of $E$, for $n$ great enough $\Gamma^\epsilon$ contains the $n$-th iterate of $\cC$.
As it is well-known (see \cite{BDV_livro}), this implies that $E\oplus F$ is a dominated splitting.
\end{proof}

Avila, Bochi, and Yoccoz \cite{ABY} already showed that (b') implies (c) in dimension 2.
Building an adapted metric, we find a simple proof that extends it to greater dimensions.

\begin{proof}[Proof that (b') implies (c)]
Assuming (b'), we will find a set $C$ satisfying the conditions (c) that has finitely many connected components, and such that these components have disjoint closures.
Let $X \times \R^d = E^u \oplus E^s$ be the dominated splitting for the universal $\Sigma$-valued cocycle.

\begin{claim}
If $x=(x_k)_{k \in \Z}$ then $E^s_x$ depends only on the future $(x_0, x_1, \ldots)$,
while $E^u_x$ depends only on the past $(\ldots, x_{-2}, x_{-1})$.
\end{claim}

\begin{proof}
For each point $x$, let us say that an $(d-i)$-plane $F$ satisfies property $\cP(x)$ if
for every $i$-plane $E$ transverse to $F$, there are $C>0$ and $\tau<1$ such that
$\|A^n(x)|F\| < C \tau^n \m(A^n(x)|E)$ for every $n \ge 0$.
Then $E^s(x)$ is the only $i$-plane that satisfies $\cP(x)$.
Since $\cP(x)$ only depends on the future, so does $E^s_x$.
Symmetrically, we see that $E^u_x$ is a function of the past.
\end{proof}

We regard the bundles of the dominated splitting as continuous maps $E^u: \Sigma^\Z \to G(i,d)$,
$E^s: \Sigma^\Z \to G(d-i,d)$.
Let $C^u$ and $C^s$ be their respective images.
These two sets are compact, and transverse (in the sense that $PC^u$ and $PC^s$ are disjoint).
Indeed, for any $x =(x_k)$ and $y=(y_k)$,
we have $E^u_x = E^u_z$ and $E^s_y = E^s_z$ where $z = (\ldots, x_{-2}, x_{-1}, y_0, y_1, \ldots)$,
so transversality follows.

Let $C^{\pitchfork s}$ be the (open) subset of $G(i,d)$ formed by the
$i$-planes that are transverse to $C^s$.
We are going to define a metric $d_\infty$ on this set.
Start fixing any Riemannian metric $d_0$ on $G(i,d)$.
Then let
$$
d_n(E,F)=\sup_{A\in\Sigma^{*n}}d(A\cdot E,A\cdot F)
\quad \text{for $E$, $F \in C^{\pitchfork s}$.}
$$
Notice that $d_n$ goes uniformly exponentially fast to zero on
compact subsets of $C^{\pitchfork s}\times C^{\pitchfork s}$.
Indeed, for any $E\in C^{\pitchfork s}$, the subbundle $X \times E$ of $X\times \R^d$ is transverse to $E^s$. Hence, by compactness of $X$ and domination, the positive iterates of $X \times E$ converge exponentially fast to $E^u$, and the speed of convergence does not depend locally on $E$.
Therefore the series $d_\infty=\sum_{i\in\N}d_n$
converges uniformly on compact subsets of $C^{\pitchfork s} \times C^{\pitchfork s}$.
In particular, $d_\infty$ is a metric on $C^{\pitchfork s}$ that
induces the same topology as $d_0$.

Notice that if $E \neq F \in C^{\pitchfork s}$ then for any $A \in \Sigma$ we have
$d_\infty(A \cdot E, A \cdot F) < d_\infty(E, F)$. We call $d_\infty$ an {\em adapted metric}.

Consider the open $\epsilon$-neighborhood $C^u_\epsilon$ of $C^u$ with respect to the $d_\infty$ metric.
Since $C^u$ is compact, $C^u_\epsilon$ has finitely many connected components. The number of connected components of $C^u_\epsilon$ decreases when $\epsilon$ increases. Hence one can choose $\epsilon$ such that the number of connected components of $C^u_\nu$ is the same as for $C^u_\epsilon$, for some $\nu>\epsilon$. Then the connected components of $C^u_\epsilon$ are isolated.
Finally, for any $\epsilon>0$ there is $0< \delta< \epsilon$ such that $\Sigma^* C^u_\epsilon \subset C^u_\delta$.
Thus $C^u_\epsilon$ satisfies the all the conditions of (c).
\end{proof}

This ends the proof of Theorem~\ref{t.char}.

\begin{proof}[Proof of the Scholium]
First, in the proof that (b') implies (c) we already showed that $C$ can be found with the required additional properties, namely its connected components $C_1, \dots, C_k$ have disjoint closures.
Now let $C'\subset P\R^d$ be the union of the sets $P\overline{C_i}$. Since $\overline{C}=\overline{C_1}\cup \dots \cup \overline{C_k}$ is strictly invariant by $\Sigma$, the set $C'$ is a also closed and strictly invariant by $\Sigma$. It clearly has a finite number of connected components. Thus $C'$ satisfies (d) and the required additional properties.
\end{proof}

\section{An example}\label{s.example}

We follow the definition of~\cite{GV} and say that a set $C\subset P\R^d=G(1,d)$ is {\em semiconvex} if for any projective line $\Delta$, the intersection of $\Delta$ and $C$ is connected.
A familiar example is given by:
$$
C = P \big\{ (u,v) \in \R^d ; \; u \in \R^i, \ v \in \R^{d-i}, \   \|v\| \le a \|u\|\big\},
\quad \text{where $a>0$;}
$$
(Notice that in this example $C$ is \emph{not} the projectivization of a convex cone in $\R^d$,
say if $i=2$ and $d=3$.)

We say that a multicone $C\subset G(i,d)$ is {\em locally semiconvex} if each of its connected components $C_k$ is such that $PC_k$ is semiconvex.

\medskip

Assume that $\Sigma$ is an $i$-dominated set,
and hence has unstable multicone of index~$i$.
It is natural to look for multicones with additional good properties.
(For example, \cite{ABY} considers so-called tight multicones.)
It is possible to prove that a locally semiconvex unstable multicone of index $i$ always exists,
provided $i$ equals $1$ or $d-1$.\footnote{This leads to interesting dichotomies between $i$-dominated sets and $i$-elliptic sets (sets $\Sigma$ that contain matrices $M_1,\dots, M_k$ such that the $i$-th and $(i+1)$-th eigenvalues of the product $M_k\dots M_1$ are complex and conjugate).}
We will not prove this fact here,
but we will show that it unfortunately does not hold without assuming $i\neq 1$, $d-1$.
More precisely, we show the following:

\begin{thm}
\label{t.example}
There is a continuous map $A\colon [0,1]\to GL(4,\R)$ such that the set $\Im(A)$ is dominated of index $2$, and such that $\Im(A)$ has no locally semiconvex unstable (nor stable) multicone of index $2$.
Moreover, these properties persists by $C^0$-perturbations of $A$.
\end{thm}

We begin with a geometrical construction:

\begin{lemma}
There are two continuous one-parameter families of lines
$L_1(t)$, $L_2(t)$  in $\R^3$, where $t$ runs in $[0,\pi]$,
with the following properties:
\begin{itemize}
\item There are four points $a\in L_1(0)$, $b \in L_2(\pi)$, $c \in L_1(\pi)$, $d\in L_2(0)$
belonging to the oriented $x$-axis $\mathbb{X}$ such that $a<b<c<d$.
\item For every $t$, $s$, the lines $L_1(t)$ and $L_2(s)$ are skew.
\end{itemize}
\end{lemma}

\begin{proof}
We take usual coordinates $x$, $y$, $z$ in the space $\R^3$.
Consider the curves
$$
\gamma_1(t) = \left(t -\sin t, \tfrac{1}{2} \sin t ,0\right)\, , \quad
\gamma_2(t) = \left(\tfrac{3\pi}{2} - t + \sin t, -\tfrac{1}{2} \sin t ,0\right) \, , \quad
t \in [0,\pi] \, ,
$$
and their endpoints $a = \gamma_1(0)$, $c=\gamma_1(\pi)$, $d=\gamma_2(0)$, $b= \gamma_2(\pi)$.
See Figure~\ref{f.curves}.
\psfrag{K1}[][l]{{$K_1$}}
\psfrag{K2}[][r]{{$K_2$}}
\psfrag{a}[][u]{{$a$}}
\psfrag{b}[][d]{{$b$}}
\psfrag{c}[][u]{{$c$}}
\psfrag{d}[][d]{{$d$}}
\begin{figure}[hbt]
\includegraphics[width=9cm]{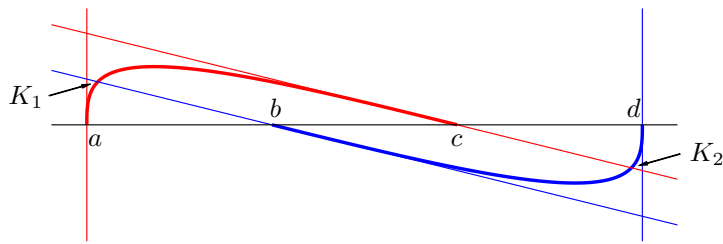}
\caption{The curves $\gamma_1$, $\gamma_2$ and their tangents at the points $a$, $b$, $c$, $d$.}
\label{f.curves} 
\end{figure}
Let $v_i(t) = \gamma_i'(t) / \|\gamma_i'(t)\|$.
Let $L_i(t)$ be the line containing the point $\gamma_i(t)$ and parallel to the vector $v_i(t) + (0,0,1)$.
Notice that no lines $L_1(t)$ and $L_2(s)$ are parallel.
Next consider the (ruled) surfaces $S_i = \bigcup_t L_i(t)$; see Figure~\ref{f.surfaces}.
We are left to prove that $S_1\cap S_2 =\emptyset$.
\begin{figure}[hbt]
\includegraphics[width=5cm]{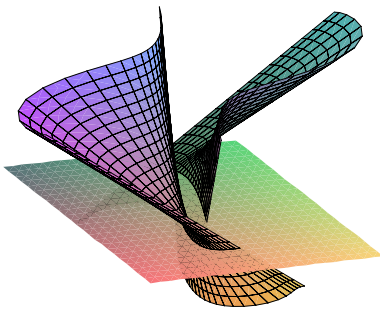}
\caption{The ruled surfaces $S_1$, $S_2$, and the plane $z=0$.}
\label{f.surfaces} 
\end{figure}

Let $S_i^+$ and $S_i^-$ be the intersections of $S_i$ with the half-spaces $\{z>0\}$ and $\{z<0\}$, respectively.
Then $S_i = S_i^- \sqcup \Im(\gamma_i) \sqcup S_i^+$.
Moreover, each $S_i^\pm$ is the graph of a function $z_i^\pm$ defined on a region $P_i^\pm$ of the $xy$ plane.
The regions $P_1^-$ and $P_2^-$ are disjoint.
On the other hand, $P_1^+ \cap P_2^+$ has two connected components $K_1$ and $K_2$; see again Figure~\ref{f.curves}.
On $K_1$ we have
$$
z_1^+ \le \mathrm{diam}(K_1) < \mathrm{dist}(b,K_1) \le z_2^+  \, ,
$$
thus the surfaces do not intersect over $K_1$.
Symmetrically, the same holds on $K_2$.
\end{proof}

We remark that the domain of definition above may be extended to some slightly larger interval $[-\epsilon, \pi+\epsilon]$
so that the two families of lines are still skew. Now by transversality of the ruled surfaces $L_i$ and the $x$-axis $\mathbb{X}$, one sees that for any $C^0$-perturbation $L'_i$ of the families $L_i$ one has:
\begin{itemize}
\item two points $a\in L_1'(t_1)$, $c\in L_1'(t_2)$ and two points $b \in L_2'(s_1)$, $d \in L_2'(s_2)$
belonging to $\mathbb{X}$ such that $a<b<c<d$.
\item For every $t$, $s$, the lines $L_1'(t)$ and $L_2'(s)$ are skew.
\end{itemize}

Now viewing $\R^3$ as the subset of points $[x:y:z:1]$ of $P\R^4$, we have the straightforward consequence:
\begin{lemma}\label{l.again}
There are two continuous one-parameter families $D_1(t)$, $D_2(s) \in G(2,4)$, $t$, $s\in [0,1]$ such that:
 \begin{itemize}
 \item every plane $D_1(t)$ is transverse to every plane of $D_2(s)$,
 \item there is a plane $P$ that contains four lines $a$, $c$ in planes of $\Im(D_1)$ and $b$, $d$ in planes of $\Im(D_2)$ such that for some cyclic order in the circle $\P(P)$ we have
$a < b < c < d < a$.
\end{itemize}
Moreover, the same properties hold for $C^0$ perturbations of $D_1$, $D_2$.
\end{lemma}

If $D_1(t)$ and $D_2(s)$ are as in the lemma, then
the sets $\Gamma_1=\bigcup_{t\in[0,1]}D_1(t)$ and $\Gamma_2=\bigcup_{t\in[0,1]}D_2(t)$ cannot be semiconvex cones.
Fix disjoint open neighborhoods $C_1 \supset \Gamma_1$ and  $C_2 \supset \Gamma_2$ in $G(2,4)$.
Now let $\lambda>1$ and define $A(t)\in GL(d,\R)$ for $t\in [0,1]$ as the linear map that coincides with
$\lambda \cdot \Id$ on $D_1(t)$ and with $\lambda^{-1} \cdot \Id$ on $D_2(t)$.
Fixing $\lambda$ large enough, we have that
$C_1$ and $C_2$ are strictly invariant by the sets $\Sigma = \{A(t)\}$ and $\{A(x)^{-1}\}$, respectively.
Any unstable multicone of index $2$ of $\Sigma$ has a connected component $C$ that contains the set $\Gamma_1$ and cannot intersect $\Gamma_2$. By Lemma~\ref{l.again}, the projective set $PC$ is not semiconvex. Hence $\Sigma$ does not admit any locally semiconvex unstable multicone of index $2$. By symmetry, the same holds replacing unstable by stable.
It is clear that these properties persist by $C^0$-perturbations of $A$. This ends the proof of Theorem~\ref{t.example}.


\end{document}